\documentclass[10pt]{amsart}
\usepackage{amssymb}
\usepackage{amsmath}
\usepackage{amsthm}
\usepackage[all]{xy}
\usepackage{mathrsfs}
\usepackage{array}
\usepackage{caption}
\usepackage{lscape}
\usepackage{color}
\usepackage{pstricks}

\include{scrload}

\newcommand{\A}{\ensuremath{\mathbb{A}}}
\newcommand{\C}{\ensuremath{\mathbb{C}}}

\newcommand{\F}{\ensuremath{\mathbb{F}}}

\newcommand{\bM}{\ensuremath{\mathbb{M}}}

\newcommand{\bP}{\ensuremath{\mathbb{P}}}

\newcommand{\R}{\ensuremath{\mathbb{R}}}

\newcommand{\Z}{\ensuremath{\mathbb{Z}}}

\newcommand{\inv}{^{-1}}

\newcommand{\into}{\hookrightarrow}

\newcommand{\rtarr}{\longrightarrow}
\newcommand{\iso}{\cong}

\newcommand{\cl}{\mathrm{cl}}
\newcommand{\Sq}{\ensuremath{\mathrm{Sq}}}

\newcommand{\mmf}{\mathit{mmf}}
\newcommand{\tmf}{\mathit{tmf}}

\DeclareMathOperator{\Hom}{Hom}

\DeclareMathOperator{\Ext}{Ext}
\DeclareMathOperator*{\colim}{colim}

\newcolumntype{C}{>{$}l<{$}}
\newcolumntype{L}{>{$}l<{$}}
\newcolumntype{|}{}
\newcolumntype{H}{>{\setbox0=\hbox\bgroup$}c<{$\egroup}@{}}

\newcommand{\altn}[1]{\mathbf{#1}}

\begin{document}
\title{The $\eta$-local motivic sphere}
\author{Bertrand J. Guillou}
\address{Department of Mathematics\\ University of Kentucky\\
Lexington, KY 40506, USA}
\email{bertguillou@uky.edu}

\author{Daniel C. Isaksen}
\address{Department of Mathematics\\ Wayne State University\\
Detroit, MI 48202, USA}
\email{isaksen@wayne.edu}
\thanks{The first author was supported by Simons Collaboration Grant 282316.
The second author was supported by NSF grant DMS-1202213.}



\begin{abstract}
We compute the $h_1$-localized cohomology of the motivic Steenrod algebra over $\C$.
This serves as the input to an Adams spectral sequence that computes the
motivic stable homotopy groups of the $\eta$-local motivic sphere.
We compute some of the Adams differentials, and we state a conjecture about
the remaining differentials. 
\end{abstract}

\date{\today}

\maketitle

\newtheorem{thm}{Theorem}[section]
\newtheorem*{thmA}{Theorem A}
\newtheorem{prop}[thm]{Proposition}
\newtheorem{lemma}[thm]{Lemma}
\newtheorem{cor}[thm]{Corollary}
\newtheorem{conj}[thm]{Conjecture}
\newtheorem{quest}{Question}
\theoremstyle{definition}
\newtheorem{rmk}[thm]{Remark}
\newtheorem{eg}[thm]{Example}
\newtheorem{defn}[thm]{Definition}
\newtheorem{notn}[thm]{Notation}

\newenvironment{pf}{\begin{proof}}{\end{proof}}

\newcounter{marker}

\section{Introduction}

Consider the Hopf map $\eta:\A^2\setminus\{0\} \rtarr \bP^1$
that takes $(x,y)$ to $[x:y]$.
In motivic homotopy theory, $\A^2\setminus\{0\}$ and $\bP^1$
are models for the motivic spheres $S^{3,2}$ and $S^{2,1}$ respectively.
Therefore, $\eta$ represents an element of the stable motivic
homotopy group $\pi_{1,1}$.

Computations of motivic stable homotopy groups share many similarities to the
classical computations, but the motivic computations also exhibit
``exotic" non-classical phenomena.  One of the first examples is that
$\eta$ is not nilpotent, i.e., $\eta^k$ is non-zero for all $k$
\cite{Morel}.

Working over $\C$ (or any algebraically closed field of characteristic zero),
we have an Adams spectral sequence for computing motivic stable
$2$-complete homotopy groups with good convergence properties
\cite{Morel2} \cite{DI} \cite{HKO}.
The $E_2$-page of this spectral sequence is the cohomology
$\Ext_A$ of the motivic Steenrod algebra $A$.
In the Adams spectral sequence, $\eta$ is detected by the element
$h_1$ of $\Ext_A$.
The failure of $\eta$ to be
nilpotent is detected by the fact that $h_1^k$ is a non-zero permanent cycle
for all $k$.

A further investigation of the motivic Adams $E_2$-page reveals a number of
other classes that survive $h_1$-localization, i.e., classes $x$ such that $h_1^k x$
is non-zero for all $k$.  The first few are $c_0$ in the $8$-stem,
$P h_1$ in the $9$-stem, $d_0$ in the $14$-stem, and $e_0$ in the $17$-stem.
A fairly predictable pattern emerges, involving classes that all map
non-trivially to the cohomology of motivic $A(2)$ \cite{IMotA(2)}.

However, in the 46-stem, a surprise occurs.  The class $B_1$ is $h_1$-local.
This element is not detected by the cohomology of motivic $A(2)$.
At this point, it has become clear that an algebraic computation of 
the $h_1$-localized cohomology $\Ext_A[h_1^{-1}]$ 
is an interesting and non-trivial problem.

The first goal of this article is calculate $\Ext_A[h_1^{-1}]$.
We will show that it is a polynomial algebra over
$\F_2 [h_1^{\pm 1}]$ on infinitely many generators.

\begin{thm}\label{thmA} The $h_1$-localized algebra $\Ext_A[h_1^{-1}]$ is a polynomial algebra over $\F_2[h_1^{\pm 1}]$ on generators $v_1^4$ and $v_n$ for $n\geq 2$,
where:
\begin{enumerate}
\item
$v_1^4$ is in the 8-stem and has Adams filtration 4.
\item
$v_n$ is in the $(2^{n+1} - 2)$-stem and has Adams filtration 1.
\end{enumerate}
\end{thm}

Although it is simple to state, this is actually a surprising answer.
The elements $d_0$, $e_0$, and $e_0 g$ are indecomposable
elements of $\Ext_A$ that are all $h_1$-local.  From consideration of the
May spectral sequence, or from the cohomology of $A(2)$, one might
expect to have the relation $e_0^3 + d_0 \cdot e_0 g = 0$.

In the terms of Theorem \ref{thmA}, $e_0^3 + d_0 \cdot e_0 g$
corresponds to $h_1^9 v_3^3 + h_1^9 v_2^2 v_4$.  Theorem \ref{thmA} says that
this expression is non-zero after $h_1$-localization, so it is
non-zero before localization as well.
The only possibility is that 
$e_0^3 + d_0 \cdot e_0 g$ equals $h_1^5 B_1$.

The relation $e_0^3 + d_0 \cdot e_0 g = h_1^5 B_1$ 
in $\Ext_A$ is hidden
in the motivic May spectral sequence.  As described in \cite{Istems},
it is tightly 
connected to the classical Adams differential $d_3(h_1 h_5 e_0) = h_1^2 B_1$
\cite[Corollary 3.6]{BJM}.

Our $h_1$-local calculation surely implies
other similarly exotic relations in higher stems.  We do not yet possess a
sufficiently detailed understanding of $\Ext_A$ in that range, so we cannot
identify any explicit examples with certainty.  However, we expect to see a 
hidden relation $e_0 (e_0 g)^2 + d_0 \cdot e_0 g^3 = h_1^5 \cdot g^2 B_1$
in $\Ext_A$ in the 91-stem.  Here, we anticipate that $e_0 g^3$ and $g^2 B_1$
are indecomposable elements.

So far, we have only discussed the entirely algebraic question of computing
$\Ext_A[h_1^{-1}]$, which informs us about the structure of $\Ext_A$.
But $\Ext_A[h_1^{-1}]$ is also 
the $E_2$-page of an Adams spectral sequence that converges to the
$2$-complete motivic stable homotopy groups of the $\eta$-local motivic sphere
$S^{0,0}[\eta^{-1}]$, which is the homotopy colimit of the sequence
\[
\xymatrix@1{
S^{0,0} \ar[r]^\eta & 
S^{-1,-1} \ar[r]^\eta &
S^{-2,-2} \ar[r]^\eta &
\cdots.
}
\]
In order to compute $\pi_{*,*} S^{0,0}[\eta^{-1}]$,
we thus only need to compute Adams differentials on $\Ext_A[h_1^{-1}]$ .
There are indeed non-trivial differentials.
In the non-local case, we know that 
$d_2(e_0) = h_1^2 d_0$ and
$d_2(e_0 g) = h_1^2 e_0^2$ \cite{Istems}.
This implies the analogous $h_1$-local differentials
$d_2(v_3) = h_1 v_2^2$ and
$d_2(v_4) = h_1 v_3^2$.

Unfortunately, we have not been able to identify all of the Adams differentials.
We expect the answer to turn out as stated in Conjecture \ref{mainConj}.

\begin{conj}\label{mainConj}
For all $k \geq 3$, there is an Adams differential
$d_2 (v_k ) = h_1 v_{k-1}^2$.
\end{conj}

Conjecture \ref{mainConj} has the following immediate consequences.

\begin{conj}
\label{mainConj2}
\mbox{}
\begin{enumerate}
\item
The $E_\infty$-page of the $h_1$-local Adams spectral sequence is
\[ \F_2[h_1^{\pm 1}][v_1^4,v_2]/v_2^2.\]
\item
The motivic stable homotopy groups of the $\eta$-local motivic sphere
are
\[
\pi_{*,*}(S^{0,0}[\eta^{-1}]) \iso \F_2[\eta^{\pm 1},\mu,\epsilon]/\epsilon^2,
\]
where $\eta$ has degree $(1,1)$;
$\mu$ has degree $(9,5)$;
and $\epsilon$ has degree $(8,5)$.
\end{enumerate}
\end{conj}

\begin{proof}
Given the Adams differentials proposed in Conjecture \ref{mainConj},
we can compute that the $h_1$-local $E_3$-page
is equal to 
$\F_2[h_1^{\pm 1}][v_1^4,v_2]/v_2^2$.
For degree reasons, there are no possible higher differentials,
so this expression is also equal to the $h_1$-local $E_\infty$-page.

There are no possible hidden extensions in $E_\infty$, so we obtain
$\pi_{*,*}(S^{0,0}[\eta^{-1}])$ immediately.
\end{proof}

\begin{rmk}
The $h_1$-localization of the element $Ph_1$ of $\Ext_A$ is $h_1 v_1^4$,
and $\mu$ is the standard notation for the element of $\pi_{9,5}$
detected by $P h_1$.
The $h_1$-localization of the element $c_0$ of $\Ext_A$ is $h_1^2 v_2$,
and $\epsilon$ is the standard notation for the element of $\pi_{8,5}$
detected by $c_0$.
\end{rmk}

We present one more consequence of Conjecture \ref{mainConj}.

\begin{thm}
\label{thmANSS}
The following are equivalent:
\begin{enumerate}
\item
Conjecture \ref{mainConj} holds.
\item
At $p=2$, the $\alpha_1$-localization of the classical Adams-Novikov spectral
sequence $E_2$-page is a free $\F_2 [\alpha_1^{\pm 1}]$-module 
with basis consisting of elements of the form $\alpha_{k/b}$
for $k = 1$ and $k \geq 3$.
\end{enumerate}
\end{thm}

Here $b$ is an integer that depends on $k$.  
The point is that $\alpha_{k/b}$ is the generator of the classical Adams-Novikov
$E_2$-page in degree $(2k-1,1)$.

\begin{proof}
Recall that $\tau$ is an element of $\pi_{0,-1}$.
Because $\tau \eta^4$ is zero, we know that the
$\eta$-localization $C\tau[\eta^{-1]}$ of the cofiber $C\tau$ of $\tau$
splits as $S^{0,0}[\eta^{-1}] \vee S^{1,-1}[\eta^{-1}]$.
Therefore,
$\pi_{*,*}(C\tau[\eta^{-1}])$ is the same as two copies of 
$\pi_{*,*}(S^{0,0}[\eta^{-1}])$.

We explain in \cite{Istems} that $\pi_{*,*} (C\tau)$
is equal to the classical Adams-Novikov $E_2$-page.
Therefore, $\pi_{*,*} (C\tau[\eta^{-1}])$
is the same as the $\alpha_1$-localization of the 
Adams-Novikov $E_2$-page.
\end{proof}

The origin of this work lies in the first author's attempt to analyze the
Adams spectral sequence beyond the 45-stem.  The $h_1$-local calculations 
discussed in this article are a helpful tool in the analysis of Adams differentials.
We expect that the $h_1$-local calculations
will continue to be a useful tool in the further analysis of Adams differentials.
For example,
our work leads us to 
anticipate an Adams differential $d_2(e_0g^3) = h_1^2(e_0g)^2$ in the 77-stem.

In this article, we are working exclusively in motivic homotopy theory
over $\C$.  A natural extension is to consider $h_1$-local and $\eta$-local
calculations over other fields.  Preliminary calculations over $\R$ show that the
picture is more complicated.  We plan to explore this in more detail
in future work.

The questions studied in this article become trivial in the classical situation,
where $h_1^4$ is zero in the cohomology of the classical Steenrod algebra,
and $\eta^4$ is zero in the classical 4-stem.  
This is consistent with the principle
that 
$\tau$-localization corresponds to 
passage from the motivic to the classical situations \cite{Istems},
and that $\tau h_1^4$ and $\tau \eta^4$ are both zero motivically.

However,
$\eta$ is not nilpotent in $\Z/2$-equivariant stable homotopy groups.
The equivariant analogues of our calculations are interesting open questions.

Our work raises the question of why Nishida's nilpotence theorem
\cite{Nishida} fails in motivic homotopy theory.
One might wonder whether $\eta$ in $\pi_{1,1}$ 
can be non-nilpotent because it has ``simplicial dimension" $0$.
However, this cannot be the full explanation since
the element $\mu$ of $\pi_{9,5}$ detected by 
$P h_1$ is also not nilpotent, and $\mu$ has simplicial dimension $4$.
In fact, the elements $\mu_{8k+1}$ of $\pi_{8k+1,4k}$ detected by $P^k h_1$
are all not nilpotent.  We expect that these are the only elements
of $\pi_{*,*}$ that fail to be nilpotent.

\subsection{Organization}

Section~\ref{sec:background} contains a review of the motivic Steenrod algebra and sets our notation. 
Section \ref{sec:CohomA} computes $\Ext_A[h_1^{-1}]$,
as stated in Theorem \ref{thmA}.
For completeness, we also discuss $\Ext_{A(2)}[h_1^{-1}]$.

Section~\ref{sec:MaySS} discusses
the same computation from the point of view of the motivic May spectral sequence. 
The point of this section is that it allows us to analyze 
in Section~\ref{sec:LocMap} the localization map
$\Ext_A\rtarr \Ext_A[h_1^{-1}]$ in detail through a range.
This leads to some hidden relations in $\Ext_A$ that are needed in \cite{Istems}.
The localization map is essential for deducing information about
Adams differentials in $\Ext_A$ from Adams differentials in $\Ext_A[h_1^{-1}]$.
We also consider the May spectral sequence and the localization map 
for $A(2)$ in Sections \ref{sec:MaySS} and \ref{sec:LocMap}.  
These sections are intended to be read in conjunction with the charts
in \cite{GI}.

Section~\ref{sec:ASS} gives some computations of Adams differentials in support of Conjecture~\ref{mainConj}.  We also discuss the role of a
speculative ``motivic modular forms" spectrum.

Much of the data for our computations, especially regarding the May spectral sequence, is given in tables to be found in Section~\ref{sec:Tables}.

\subsection*{Acknowledgements}

The authors thank Haynes Miller for a conversation that led to the proof
of Theorem \ref{thmA}.

\section{Background}\label{sec:background}

We continue with notation from \cite{Istems} as follows:
\begin{enumerate}
\item
$\bM_2$ is the motivic cohomology of $\C$ with $\F_2$ coefficients.
\item
$A$ is the mod 2 motivic Steenrod algebra over $\C$,
and $A_{*,*}$ is its dual.
\item
$A(n)$ is the $\bM_2$-subalgebra of $A$ generated by
$\Sq^1$, $\Sq^2$, $\Sq^4$, \ldots, $\Sq^{2^n}$,
and $A(n)_*$ is its dual.
\item
$\Ext_A$ is the trigraded ring $\Ext_A(\bM_2,\bM_2)$.
\item
More generally, $\Ext_B$ is the trigraded ring $\Ext_B(\bM_2, \bM_2)$
for any Hopf algebra $B$ over $\bM_2$.
\item
$A_{\cl}$ is the mod 2 classical Steenrod algebra,
and $A^{\cl}_*$ is its dual.
\item
$A(n)_{\cl}$ is the $\bM_2$-subalgebra of $A_{\cl}$ generated by
$\Sq^1$, $\Sq^2$, $\Sq^4$, \ldots, $\Sq^{2^n}$,
and $A(n)^{\cl}_*$ is its dual.
\item
$\Ext_{A_\cl}$ is the bigraded ring $\Ext_{A_{\cl}}(\F_2,\F_2)$.
\item
More generally, $\Ext_B$ is the trigraded ring $\Ext_B(\F_2, \F_2)$
for any Hopf algebra $B$ over $\F_2$.
\end{enumerate}

The following two  theorems of Voevodsky are the starting points
of our calculations.

\begin{thm}[\cite{V1}]
$\bM_2$ is the bigraded ring $\F_2[\tau]$, where
$\tau$ has bidegree $(0,1)$.
\end{thm}

Our main object of study will be a localization of $\Ext_A$.
It will be more convenient for us to work with the dual $A_{*,*} = \Hom_{\bM_2}(A,\bM_2)$.

\begin{thm}\cite{V2} \cite[Theorem 12.6]{V3} 
The dual motivic Steenrod algebra $A_{*,*}$ is generated as an $\bM_2$-algebra by 
$\xi_i$ and $\tau_i$,
of degrees $(2(2^i-1),2^i-1)$ and $(2^{i+1}-1,2^i-1)$ respectively,
subject to the
relations
\[ \tau_i^2 = \tau \xi_{i+1}.\]
The coproduct is given on the generators by the following formulae, in which $\xi_0 = 1$:
\[ \Delta(\tau_k) = \tau_k\otimes 1 + \sum_i \xi_{k-i}^{2^i}\otimes \tau_i\]
\[ \Delta(\xi_k) = \sum_i  \xi_{k-i}^{2^i} \otimes \xi_i.\]
\end{thm}

\begin{rmk}
The quotient $A_{*,*}/\tau=A_{*,*}\otimes_{\bM_2}\F_2$ is analogous to the 
odd-primary classical dual Steenrod algebra, in the sense that there is an
infinite family of exterior generators $\tau_i$ and an infinite family
of polynomial generators $\xi_i$.  On the other hand,
the localization $A_{*,*}[\tau^{-1}]$ is analogous to the
mod 2 classical dual Steenrod algebra, which has only polynomial generators
$\tau_i$.
\end{rmk}

\subsection{$\Ext$ groups}

We are interested in computing a localization of
$\Ext_A$.
Before localization, this is a trigraded object.  In \cite{Istems}, classes in $\Ext_A$ are described in degrees of the
form $(s,f,w)$, where:
\begin{enumerate}
\item
$f$ is the Adams filtration, i.e., the homological degree.
\item
$s+f$ is the internal degree, corresponding to the
first coordinate in the bidegrees of $A$.
\item
$s$ is the stem, i.e., the internal degree minus
the Adams filtration.
\item
$w$ is the motivic weight.
\end{enumerate}

In the cobar complex,
$\xi_1$ represents an element $h_1$ of $\Ext_A$
in degree $(1,1,1)$.  Because we will invert $h_1$,
it is convenient to choose a new grading that is more $h_1$-invariant.
Except where otherwise noted, we will use the grading
$(t, f, c)$, where:
\begin{enumerate}
\item
$t = s - w$ is the Milnor-Witt stem.
\item
$f$ is the Adams filtration.
\item
$c = s +f - 2w$ is the Chow degree \cite{Istems}, which turns out to be a 
convenient grading for calculational purposes.
\end{enumerate}

The terminology ``Milnor-Witt stem" arises from the work of Morel
\cite{Morel}, which describes the motivic stable homotopy groups
$\pi_{s,w}$ with $s-w=0$ in terms of Milnor-Witt $K$-theory.

The terminology ``Chow degree" arises from the fact that the 
grading $s+f-2w$ is a natural index from the higher Chow group perspective
on motivic cohomology \cite{Bloch}.

\section{The $h_1$-local cohomology of $A$}
\label{sec:CohomA}

The goal of this section is to compute $\Ext_A[h_1^{-1}]$ explicitly.
We will accomplish this by
expressing $A_{*,*}$ as a series of extensions of smaller Hopf algebras.

\begin{defn}
\label{defn:Bi}
For each $i\geq -1$, let $B_i$ be the subalgebra of $A_{*,*}$
generated by the elements $\xi_k$ and also by $\tau_0,\dots,\tau_i$. 
\end{defn}

In particular, $B_{-1}$ is a polynomial $\bM_2$-algebra on the elements $\xi_k$.

\begin{lemma}
\label{lem:B-1}
$\Ext_{B_{-1}}[h_1^{-1}]$ is isomorphic to
$\bM_2[h_1^{\pm 1}]$.
\end{lemma}

\begin{proof}
We have an isomorphism 
$B_{-1} \rtarr A_*^\cl\otimes_{\F_2}\bM_2$
of Hopf algebras.
Under this mapping, the element $h_1$ in $\Ext_{B_{-1}}$
corresponds to $h_0$ in $\Ext_{A_{\cl}}$.
Adams's vanishing line of slope 1 \cite{A} implies that
$\Ext_{A_{\cl}}[h_0^{-1}]$ is isomorphic to $\F_2[h_0^{\pm 1}]$.
%
\end{proof}

We proceed to compute $\Ext_{B_n}[h_1^{-1}]$ inductively via a Cartan-Eilenberg spectral sequence 
\cite[\S XVI.6]{CE}, \cite[A1.3.14]{R}
for the extension of Hopf algebras
\[ B_{n-1}\rtarr B_n \rtarr E(\tau_n),\]
where $E(\tau_n)$ is an exterior algebra on the generator $\tau_n$.
The extension is cocentral, so the spectral sequence takes the form
\[ E_2 \iso \Ext_{B_{n-1}}\otimes \Ext_{E(\tau_n)} \iso \Ext_{B_{n-1}}[v_n] \Rightarrow \Ext_{B_n}.\] 
The $E_2$-page of this spectral sequence has four gradings: 
three from $\Ext_{B_n}$ and one additional Cartan-Eilenberg grading associated
with the filtration involved in construction of the spectral sequence.
However, we will suppress the Cartan-Eilenberg grading because we won't need
it for bookkeeping purposes.

The class $v_n$ has degree $(2^n-1,1,1)$ in the $E_2$-term.
The differentials take the form
\[ d_r:E_r^{(t,f,c)} \rtarr E_r^{(t-1,f+1,c)}.\]

\begin{prop}\label{MainExtProp} 
\mbox{}
\begin{enumerate}
\item $\Ext_{B_0}[h_1^{-1}] \iso \bM_2[h_1^\pm,v_0]$.
\item $\Ext_{B_1}[h_1^{-1}] \iso \F_2[h_1^\pm,v_1^4]$.
\item
$\Ext_{B_n}[h_1^{-1}] \iso \F_2[h_1^\pm,v_1^4,v_2,\dots,v_n]$
for all $n \geq 2$.
\end{enumerate}
In each case, $v_i$ has degree $(2^i-1,1,1)$.
\end{prop}

\begin{pf}
When $n=0$, Lemma \ref{lem:B-1} says that
the Cartan-Eilenberg spectral sequence takes the form
\[ E_2 \iso \bM_2[h_1^{\pm 1}][v_0] \Rightarrow \Ext_{B_0}[h_1^{-1}].\]
Since $\tau_0$ is primitive, the spectral sequence collapses at $E_2$. 
We conclude that $\Ext_{B_0}[h_1^{-1}]\iso \bM_2[h_1^{\pm 1},v_0]$ with $v_0$ in degree $(0,1,1)$.

Taking now $n=1$, 
the computation of $\Ext_{B_0}[h_1^{-1}]$ in the previous paragraph tells
us that 
the spectral sequence takes the form
\[ E_2 \iso \bM_2[h_1^{\pm 1},v_0,v_1] \Rightarrow \Ext_{B_1}[h_1^{-1}],\]
with $v_0$ in degree $(0,1,1)$ and $v_1$ in degree $(1,1,1)$.
The coproduct formula 
\[\Delta(\tau_1) = \tau_1\otimes 1 + \xi_1\otimes \tau_0 + 1\otimes \tau_1\]
 gives rise to the differential $d_2(v_1) = h_1v_0$. 
It follows that $E_3 \iso \bM_2[h_1^\pm, v_1^2]$. There is next a differential $d_3(v_1^2) = \tau h_1^3$, which can be verified by the cobar complex calculation
\[
d( \tau_1|\tau_1 | \xi_1 + \xi_1 | \tau_0\tau_1 | \xi_1 + \tau_1\xi_1 | \tau_0 | \xi_1 + \xi_1^2 | \tau_0 | (\tau_1+\tau_0\xi_1)) = \tau \xi_1 | \xi_1 | \xi_1 | \xi_1.
\]
The class $v_1^4$ in degree $(4,4,4)$ cannot support any higher differentials for degree reasons, and we have
\[ \Ext_{B_1}[h_1^{-1}] \iso E_\infty=E_4  \iso \F_2[h_1^\pm,v_1^4].\]

For $n\geq 2$, 
the argument is by induction, using a Cartan-Eilenberg spectral sequence at every turn. Each of these spectral sequences collapses at $E_2$ since 
there are no possible values for differentials on $v_n$.
\end{pf}

\begin{thm}\label{ExtAThm}
\[ \Ext_A[h_1^{-1}] \iso \F_2[h_1^\pm,v_1^4,v_2,v_3,\dots],
\]
where $h_1$ has degree $(0,1,0)$;
$v_1^4$ has degree $(4,4,4)$; and
$v_n$ has degree $(2^n-1,1,1)$ for $n \geq 2$.
\end{thm}

\begin{proof}
Since $A$ is $\colim B_n$,
$\Ext_A$ equals $\colim \Ext_{B_n}$.
The calculation follows from part (3) of Proposition \ref{MainExtProp}.
\end{proof}

\begin{rmk}
Theorem \ref{ExtAThm} implies that 
the part of $\Ext_A [h_1^{-1}]$ in Chow degree zero is equal to
$\F_2[h_1^{\pm 1}]$.
Another more direct argument for this observation uses the
isomorphism between $\Ext_{A_\cl}$ and the Chow degree zero part of
$\Ext_A$ \cite{Istems}.  This isomorphism takes classical elements of degree 
$(s,f)$ to motivic elements of degree $(s,f,0)$. The classical calculation 
$\Ext_{A_\cl} [h_0^{-1} ] =\F_2 [h_0^{\pm 1}]$ corresponds to the Chow degree
zero part of $\Ext_A [h_1^{-1}]$.
\end{rmk}

\subsection{The $h_1$-local cohomology of $A(2)$}\label{sec:CohomA(2)}
For completeness, we will also explicitly calculate
$\Ext_{A(2)}[h_1^{-1}]$, where $A(2)$ is the $\bM_2$-subalgebra 
of $A$ generated by $\Sq^1$, $\Sq^2$, and $\Sq^4$. 

Dual to the inclusion $A(2)\into A$ is a quotient map
\[ A_{*,*} \rtarr A(2)_{*,*} \iso \bM_2[\xi_1,\xi_2,\tau_0,\tau_1,\tau_2]/(\xi_1^4,\xi_2^2,\tau_0^2=\tau\xi_1,\tau_1^2=\tau\xi_2,\tau_2^2)\]
We filter $A(2)_{*,*}$ by sub-Hopf algebras
\[ B_{-1}(2)\subseteq B_0(2) \subseteq B_1(2)\subseteq A(2)_{*,*}\]
where $B_{-1}(2)$ is the subalgebra generated by $\xi_1$ and $\xi_2$;
$B_0(2)$ is generated by $\xi_1$, $\xi_2$, and $\tau_0$;
and $B_1(2)$ is generated by $\xi_1$, $\xi_2$, $\tau_0$, and $\tau_1$.
The notation is analogous to the notation in Definition \ref{defn:Bi}.

\begin{lemma}
\label{lem:B-1(2)}
$\Ext_{B_{-1}(2)}[h_1^{-1}]$ is isomorphic to $\bM_2[h_1^{\pm 1},a_1]$,
where $a_1$ has degree $(4,3,0)$.
\end{lemma}

\begin{proof}
We have an isomorphism 
$B_{-1}(2) \rtarr A(1)_*^\cl\otimes_{\F_2} \bM_2$.
Under this mapping, the element $h_1$ in $\Ext_{B_{-1}(2)}$
corresponds to
$h_0$ in $\Ext_{A(1)_{\cl}}$.
It is well-known that 
$\Ext_{A(1)_{\cl}}[h_0^{-1}]$ is isomorphic to $\F_2 [ h_0^{\pm 1}, a^{\cl} ]$, where $a^{\cl}$ has degree $(4,3)$
(for example, see \cite[Theorem 3.1.25]{R}).
\end{proof}

\begin{prop}\label{CohomBi(2)} 
\mbox{}
\begin{enumerate}
\item $\Ext_{B_0(2)}[h_1^{-1}]\iso \bM_2[h_1^{\pm},a_1,v_0]$.
\item $\Ext_{B_1(2)}[h_1^{-1}] \iso \F_2[h_1^\pm,a_1,v_1^4]$.
\item  $\Ext_{A(2)}[h_1^{-1}]  \iso \F_2[h_1^\pm,a_1,v_1^4,v_2]$.
\end{enumerate}
In each case, $a_1$ has degree $(4,3,0)$;
$v_0$ has degree $(0,1,1)$;
$v_1^4$ has degree $(4,4,4)$;
and $v_2$ has degree $(3,1,1)$.
\end{prop}

\begin{proof}
The proof 
uses a series of Cartan-Eilenberg spectral sequences as in 
the proof of Proposition~\ref{MainExtProp}, given 
Lemma \ref{lem:B-1(2)} as the starting point.
\end{proof}

\begin{rmk} The classes $a_1$, $v_1^4$, and $h_1^2 v_2$ correspond respectively 
to the classes $u$, $P$, and $c$ in \cite[Theorem~4.13]{IMotA(2)}.
\end{rmk}

\begin{rmk}
Using the structure of $\Ext_{A(2)_{\cl}}[h_0^{-1}]$,
similar arguments show that
\[ 
\Ext_{A(3)}[h_1^{-1}] \iso \F_2[h_1^{\pm 1},g,b_{31},v_1^4,v_2,v_3],
\]
where $g$ has degree $(8,4,0)$ and 
$b_{31}$ has degree $(12,2,0)$.
Using the structure of $\Ext_{A(3)_{\cl}}[h_0^{-1}]$, 
\[ 
\Ext_{A(4)}[h_1^{-1}] \iso \F_2[h_1^{\pm 1},g^2,\Delta_1,b_{41},v_1^4,v_2,v_3,v_4], \]
where $g^2$ has degree $(16,8,0)$,
$\Delta_1$ has degree $(24,4,0)$, and
$b_{41}$ has degree $(28,2,0)$.
\end{rmk}

\section{The $h_1$-local motivic May spectral sequence}\label{sec:MaySS}

Although Theorem \ref{ExtAThm} gives a complete description of $\Ext_A[h_1^{-1}]$,
it unfortunately tells us very little about the localization map
$\Ext_A \rightarrow \Ext_A[h_1\inv]$.  The problem is that the proof 
of Theorem \ref{ExtAThm} is incompatible with the motivic May spectral 
sequence approach to $\Ext_A$, as carried out in \cite{Istems}.

A detailed understanding of the localization map allows for the transfer
of information from the well-understood $\Ext_A[h_1^{-1}]$ to the
much more complicated $\Ext_A$.
In this section we carry out the computation of the
$h_1$-localized motivic May spectral sequence.
This will allow us to obtain information about the localization map in 
Section \ref{sec:LocMap}.

We recall the details of the motivic May spectral sequence from \cite{DI}. This
spectral sequence has four gradings: three from $\Ext_A$ and one additional
May grading associated with the filtration involved in construction of the
spectral sequence.  We will grade this spectral sequence in the form
$(m-f, t, f, c)$, where $m$ is the May grading, $f$ is the Adams filtration,
$t = s-w$ is the Milnor-Witt stem, and $c = s+f-2w$ is the Chow degree.
 
The $E_1$-page is a polynomial algebra over $\bM_2$ on generators 
$h_{ij}$ for $i\geq 1$, $j\geq 0$, where
\begin{enumerate}
\item $h_{i0}$ has degree $(i-1, 2^{i-1}-1, 1, 1)$.
\item $h_{ij}$ has degree $(i-1, 2^{j-1}(2^i-1) - 1, 1, 0)$ for $j>0$.
\end{enumerate}
Note in particular that $h_{i0}$ has Chow degree $1$,
while $h_{ij}$ has Chow degree $0$ for $j > 0$.
In a sense, this wrinkle in the gradings is the primary source of 
``exotic" motivic phenomena that do not appear in the classical situation.

The $d_1$-differential is given by the classical formula
\[ d_1(h_{ij}) = \sum_{0<k<i} h_{kj}h_{i-k,j+k}.\]

\subsection{The $h_1$-local $E_1$-term}\label{sec:AMayE1}
Consider the $h_1$-localization $E_1[h_1^{-1}]$ of the May $E_1$-page.  
In order to simplify the calculation, we introduce the following notation.

\begin{defn}
In $E_1[h_1^{-1}]$, define
\begin{enumerate}
\item
$h_{n0}'$ to be $h_{n0} + h_1^{-1} h_{20}h_{n-1,1}$.
\item
$h_{n-1,2}'$ to be $h_{n-1,2}+h_1^{-1}\sum_{k=2}^{n-1} h_{k,1}h_{n-k,k}$.
\end{enumerate}
\end{defn}

We may replace the algebra generators  $h_{n0}$ and $h_{n-1,2}$
by $h_{n0}'$ and $h_{n-1,2}'$  to obtain another set of algebra generators
for $E_1[h_1^{-1}]$ 
that will turn out to be calculationally convenient.

\begin{defn}
\label{defn:F1-G1}
\mbox{}
\begin{enumerate}
\item
Let $F_1$ be the $\bM_2[h_1^{\pm 1}]$-subalgebra of $E_1[h_1^{-1}]$
on polynomial generators $h_0$, $h_{20}$, $h_{n1}$ for all $n \geq 2$,
and $h'_{n2}$ for all $n \geq 1$.
\item
Let $G_1$ be the 
$\F_2$-subalgebra of $E_1[h_1^{-1}]$
on polynomial generators $h_{n0}'$ for $n \geq 3$ and
$h_{ij}$ for $i \geq 1$ and $j \geq 3$.
\end{enumerate}
\end{defn}

Note that $F_1$ is a differential graded subalgebra of $E_1[h_1^{-1}]$
since 
$d_1(h_{20})=h_1h_0$ and 
$d_1(h_{n1})= h_1h_{n-1,2}'$.
The generators of $F_1$ are indicated in the figure below as the elements
that are outside of the shaded region.

Note also that $G_1$ is a differential graded subalgebra of $E_1[h_1^{-1}]$
because 
$d_1(h_{n,0}') = \sum_{j=3}^{n-1} h_{n-j,j}h_{j,0}'$ 
if $n\geq 3$.
The generators of $G_1$ are indicated in the figure below as the 
elements in the shaded regions.

\begin{prop}
\label{prop:E1-split} 
$E_1[h_1\inv]$ splits as a tensor product
$F_1 \otimes_{\F_2} G_1$.
\end{prop}

\begin{proof}
This follows immediately from the definitions, using that
$h_{n0}'$ and $h_{n2}'$ can be used as algebra generators in place
of $h_{n0}$ and $h_{n2}$.
\end{proof}

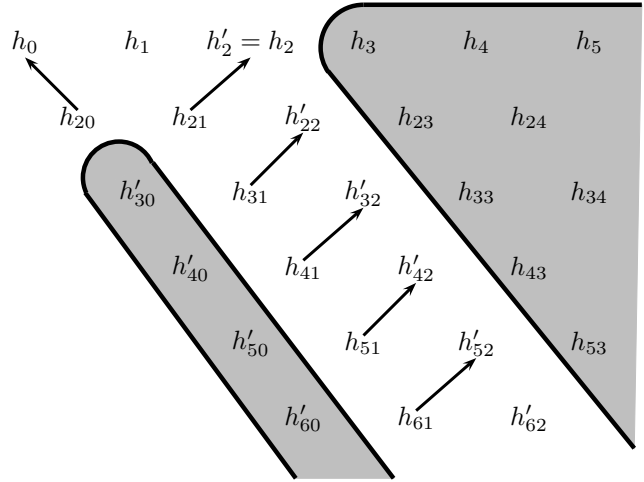
\psset{linewidth=0.4mm}    

\begin{pspicture}(0,0)(10,7)
\pspolygon[linewidth=0,fillstyle=solid,linecolor=lightgray,fillcolor=lightgray](1.8,4)(4.6,0.2)(5.9,0.2)(2.7,4.4)
\pspolygon[linewidth=0,fillstyle=solid,linecolor=lightgray,fillcolor=lightgray](5.5,6.5)(9.2,6.5)(9.1,0.6)(5.02,5.63)
\rput(1,6){$h_0$}
\rput(2.5,6){$h_1$}
\rput(4,6){$h_2'=h_2$}
\rput(5.5,6){$h_3$}
\rput(7,6){$h_4$}
\rput(8.5,6){$h_5$}
\rput(1.7,5){$h_{20}$}
\psline{->}(1.7,5.1)(1,5.8)
\rput(3.2,5){$h_{21}$}
\psline{->}(3.2,5.1)(4,5.8)
\rput(4.7,5){$h_{22}'$}
\rput(6.2,5){$h_{23}$}
\rput(7.7,5){$h_{24}$}
\rput(2.5,4){$h_{30}'$}
\rput(4,4){$h_{31}$}
\psline{->}(4,4.1)(4.7,4.8)
\rput(5.5,4){$h_{32}'$}
\rput(7,4){$h_{33}$}
\rput(8.5,4){$h_{34}$}
\rput(3.2,3){$h_{40}'$}
\rput(4.7,3){$h_{41}$}
\psline{->}(4.7,3.1)(5.5,3.8)
\rput(6.2,3){$h_{42}'$}
\rput(7.7,3){$h_{43}$}
\rput(4,2){$h_{50}'$}
\rput(5.5,2){$h_{51}$}
\psline{->}(5.5,2.1)(6.2,2.8)
\rput(7,2){$h_{52}'$}
\rput(8.5,2){$h_{53}$}
\rput(4.7,1){$h_{60}'$}
\rput(6.2,1){$h_{61}$}
\psline{->}(6.2,1.1)(7,1.8)
\rput(7.7,1){$h_{62}'$}
\psset{linewidth=0.6mm}    
\psline(1.8,4)(4.6,0.2)
\psline(2.7,4.4)(5.9,0.2)
\psarc[fillstyle=solid,fillcolor=lightgray](2.25,4.2){0.48}{24}{205}
\psline(5.5,6.5)(9.2,6.5)
\psline(5.02,5.63)(9.1,0.6)
\psarc[fillstyle=solid,fillcolor=lightgray](5.5,5.93){0.57}{90}{215}
\end{pspicture}

We will now show that the perhaps obscurely defined subalgebra $G_1$
is isomorphic to the familiar classical May $E_1$-page.

\begin{prop} 
\label{prop:cl-G1}
Let $E_1^{\cl}$ be the $E_1$-page of the classical May spectral sequence.
Consider the algebra map $S:E_1^{\cl}\rtarr G_1$ determined by
\begin{enumerate}
\item
$S(h_{n0}) = h_{n+2,0}'$ for all $n \geq 1$.
\item
$S(h_{nk}) = h_{n,k+2}$ for all $n \geq 1$ and $k \geq 1$.
\end{enumerate}
The map $S$ is an isomorphism of differential graded algebras.
\end{prop}

\begin{pf}
We need to check that $S$ preserves the May $d_1$ differential.
This is a straightforward computation.
\end{pf}

\subsection{The $h_1$-local $E_2$-term}
We now have a good understanding of $E_1[h_1^{-1}]$
from Propositions \ref{prop:E1-split} and \ref{prop:cl-G1}.
Next we analyze the $h_1$-localization
$E_2[h_1^{-1}]$ of the motivic May $E_2$-page.
Since localization is exact, $E_2[h_1^{-1}]$ is isomorphic
to the cohomology of the differential graded algebra
$E_1[h_1^{-1}]$.

\begin{prop}
$E_2[h_1^{-1}]$ is isomorphic to 
\[ \bM_2[h_1^{\pm}][b_{20},b_{21},b_{31},b_{41},\dots] \otimes_{\F_2} G_2,\]
where 
$G_2$ is the cohomology of the differential graded algebra $G_1$
from Definition \ref{defn:F1-G1}.
\end{prop}

\begin{proof}
This follows from the splitting given in Proposition \ref{prop:E1-split}.
The calculation of the cohomology of the subalgebra $F_1$ from
Definition \ref{defn:F1-G1} is straightforward.
\end{proof}

Because of Proposition \ref{prop:cl-G1},
we know that $G_2$ is isomorphic to the classical
May $E_2$-page.
May's original calculation \cite{May} of the classical $E_2$-term in stems below $156$ gives us complete understanding of $G_2$ in a much larger range
because the map $S$ from Proposition \ref{prop:cl-G1} approximately
quadruples degrees.

Generators and relations for $E_2[h_1^{-1}]$ 
up to the Milnor-Witt 66-stem
can be found in Tables~\ref{E2GenTable} and \ref{E2RelTable},
where we use the following notation.

\begin{notn}
For an element $x$ in the classical May spectral sequence,
let $\altn{x}$ be the element $S(x)$ of the
$h_1$-localized motivic May spectral sequence from Proposition~\ref{prop:cl-G1}.
\end{notn}

According to this notation, the classical element $c_0 = h_1h_0(1)$ may be written as $c_0 = h_1^2\altn{h_0}$, so that the elements $c_0$ and $\altn{h_0}$ are practically interchangeable. As the primary goal of our computation is to relate the answer to $\Ext_A$, we will most often choose to work with $c_0$. 
However, we will opt instead to use $\altn{h_0}$ when it illuminates the 
structure of the $h_1$-localized motivic May spectral sequence,
especially in Section \ref{subsctn:May-diff}.

See also Table \ref{E4NotationTable} for additional notation used
on later pages. The names of many classes in the May spectral sequence have been chosen to agree with the notation of \cite{Istems}, and we denote the remaining new classes by $y_n$.

\subsection{The $h_1$-local May differentials}
\label{subsctn:May-diff}

We now understand $E_2[h_1^{-1}]$ in a very large range of dimensions.
The next step is to compute the higher differentials to obtain 
$E_\infty[h_1^{-1}]$.

\begin{prop}
\label{prop:May-d2}
Table \ref{E2GenTable} gives the values of the May $d_2$ differential on 
the multiplicative generators of $E_2[h_1^{-1}]$ through the Milnor-Witt 66-stem.
\end{prop}

\begin{proof}
As discussed in \cite[\S5]{DI}, the $d_2$ differential of the motivic May spectral is easy to determine from the classical $d_2$ differential; 
the formulas are the same, except that powers of $\tau$ must sometimes be inserted to balance the weights. This, combined with the fact that $h_1$-localization kills the classes $h_0$ and $h_2$, leads to the values in Table~\ref{E2GenTable}.
\end{proof}

The values of the May $d_2$ differential given in Table \ref{E2GenTable}
allow us to compute $E_4[h_1^{-1}]$ directly.
A chart of $E_4[h_1^{-1}]$ through the Milnor-Witt 66-stem is given in
\cite{GI}.  

Now we proceed to the higher May differentials and the higher
$h_1$-localized pages of the motivic May spectral sequence.

\begin{prop} 
For $r \geq 4$,
some values for the May $d_r$ differentials are given in 
Tables~\ref{d4MayTable}--\ref{higherMayTable}.
The May $d_r$ differentials are zero on all other multiplicative generators
of $E_r[h_1^{-1}]$ through the Milnor-Witt 66-stem.
\end{prop}

\begin{pf}
Most of the differentials are forced by the known structure
of $\Ext_A[h_1^{-1}]$ given in Theorem~\ref{ExtAThm}.
For example, Theorem~\ref{ExtAThm} implies that 
$h_1^k v_3 v_4$ are the only non-zero elements in the Milnor-Witt 22-stem
with Chow degree 2.  We know that $h_1^k e_0 \cdot e_0 g$ survive the
May spectral sequence to detect these elements.
However, the element $\phi$ in the May $E_4$-page is also in the
Milnor-Witt 22-stem with Chow degree 2.  Therefore, it cannot survive the
May spectral sequence.  The only possibility is that
$d_6(\phi)$ equals $h_1 c_0^2 h_5$.

There are a handful of  more difficult cases, which are handled individually
in the following lemmas.
\end{pf}

\begin{lemma}
\label{lem:d4}
\mbox{}
\begin{enumerate}
\item
$d_4(\altn{P}) = \altn{h_0^4h_3}$.
\item
$d_4(\altn{\Delta}) = \altn{h_4P}$.
\end{enumerate}
\end{lemma}

\begin{pf}
For the first formula,
we use Nakamura's squaring operations \cite{Nak} in the May spectral sequence
to compute that 
\[\begin{split}
 d_4(\altn{P}) &= d_4 (\Sq_0(\altn{b_{20}})) = \Sq_1 d_2(\altn{b_{20}}) = \Sq_1(\altn{h_0^2h_2}) \\ 
 &= \Sq_1(\altn{h_0^2})\Sq_0(\altn{h_2}) + \Sq_0(\altn{h_0^2})\Sq_1(\altn{h_2}) = 0 + \altn{h_0^4}\altn{h_3}.
 \end{split}\]

The proof for the second formula is similar:
\[\begin{split}
 d_4(\altn{\Delta}) &= d_4 (\Sq_0(\altn{b_{30}})) = \Sq_1 d_2(\altn{b_{30}}) = \Sq_1(\altn{h_3b_{20}}) \\ 
 &= \Sq_1(\altn{h_3})\Sq_0(\altn{b_{20}}) + \Sq_0(\altn{h_3})\Sq_1(\altn{b_{20}}) =  \altn{h_4}\altn{P}+ 0.
 \end{split}\]
\end{pf}

\begin{lemma}
\label{lem:d6}
$d_6(c_0g\phi\Delta_1) = h_1^7 D_3' \Delta_1$. 
\end{lemma}

\begin{proof}
The class $D_3'\Delta_1$ cannot survive by Theorem~\ref{ExtAThm}. 
There are no classes for it to hit, and the only other differential that could possibly hit $h_1^7 D_3'\Delta_1$ is $d_{10}(h_1^{-8} c_0g^4\phi)$. But $d_{10}(g^4\phi) = h_1^{13}gy_{45}$, so $c_0g^4\phi$ is a $d_{10}$-cycle.
\end{proof}

\begin{lemma}
\label{lem:d8}
\mbox{}
\begin{enumerate}
\item
$d_8(\Delta_1\altn{P}) = h_1^{-6} c_0^2 e_0^2h_6$.
\item
$d_8(\altn{P^2})= \altn{h_0^8h_4}$.
\end{enumerate}
\end{lemma}

\begin{proof}
Start with the relation $c_0^2 \cdot\Delta_1 \phi = h_1^{2} e_0 \cdot \Delta_1 B_1$.
We will show in Lemma \ref{lem:d10} that
$\Delta_1 B_1$ is a permanent cycle.  Therefore,
$c_0^2 \Delta_1 \phi$ is a permanent cycle.
However, $\Delta_1 \phi$ cannot survive by Theorem \ref{ExtAThm},
so we must have 
$d_{10}(\Delta_1\phi) = h_1^3 e_0^2 h_6$.
Since $d_{10}(c_0^2 \Delta_1 \phi)$ must be zero,
it follows that $c_0^2 e_0^2 h_6 $ must be zero
in $E_{10}[h_1^{-1}]$.
The only possibility is that
$d_8 (\Delta_1 \altn{P} ) = h_1^{-6} c_0^2 e_0^2 h_6$.
This establishes the first formula.

For the second formula,
we use Nakamura's squaring operations \cite{Nak} as in the proof of
Lemma \ref{lem:d4} to compute
\[ \begin{split}
d_8(\altn{P^2}) &=d_8(\Sq_0(\altn{P})) = \Sq_1 d_4(\altn{P}) = 
\Sq_1( \altn{h_0^4h_3}) \\
&= \altn{h_0^8h_4} + \Sq_1(\altn{h_0^4})\altn{h_3^2} = 
\altn{h_0^8h_4}.  
\end{split}\]
\end{proof}

\begin{lemma}
\label{lem:d10}
$d_{10}(\Delta_1 B_1) = 0$.
\end{lemma}

\begin{proof}
We know that $d_{16}(e_0g^4) = h_1^{16} e_0h_6$, yet 
$c_0^2\cdot e_0g^4 = h_1^{2} e_0^2 \cdot e_0g^3$ is a permanent cycle. 
So $c_0^2e_0 h_6$ must be zero in $E_{16}[h_1^{-1}]$. 
The only possibilities are that 
$d_{10}(\Delta_1B_1) = h_1 c_0^2e_0h_6$
or $d_{14}(g^3B_1) = h_1^{9} c_0^2e_0 h_6$.

In the notation of Theorem \ref{ExtAThm},
$c_0$, $e_0$, $e_0 g$ and $e_0 g^3$ correspond to 
$h_1^2v_2$, $h_1^3 v_3$, $h_1^7 v_4$, and $h_1^{15} v_5$ respectively.
Since these elements are algebraically independent, we conclude that
the relation 
$ h_1^2e_0 (e_0 g)^2 + c_0^2 \cdot e_0 g^3 = 0$
in $E_\infty[h_1^{-1}]$ must be resolved by
\[
h_1^2 e_0 (e_0 g)^2 + c_0^2 \cdot e_0 g^3 = h_1^7 g^2 B_1
\]
in $\Ext_A[h_1^{-1}]$.

Similarly, the relation
$e_0^2 \cdot e_0g^3 + (e_0g)^3 = 0$ in $E_\infty[h_1^{-1}]$
must be resolved in $\Ext_A[h_1^{-1}]$ by 
\[
e_0^2 \cdot e_0g^3 + (e_0g)^3 = x,
\]
where $x$ is either $h_1^{13} \Delta_1 B_1$ or $h_1^5 g^3 B_1$.
Multiply the first hidden relation by $e_0^2$,
multiply the second hidden relation by $c_0^2$, and add to obtain
\[
h_1^2 e_0^3 (e_0 g)^2 + c_0^2 (e_0 g)^3 = h_1^7 e_0^2 \cdot g^2 B_1 + c_0^2 x.
\]

Again using Theorem \ref{ExtAThm},
the left side of this last relation is not zero, so the right side is also
not zero.
This implies that $x$ cannot be $h_1^5 g^3 B_1$, since
$h_1^7 e_0^2 \cdot g^2 B_1$ equals $h_1^5 c_0^2\cdot g^3 B_1$ in the 
May spectral sequence with no possible hidden extension.

Therefore, $x$ must be $h_1^{13} \Delta_1 B_1$, and
$\Delta_1 B_1$ must survive the May spectral sequence.
\end{proof}

\subsection{$E_\infty[h_1^{-1}]$ and $\Ext_A[h_1^{-1}]$}

The May differentials given in Section \ref{subsctn:May-diff} allow us to
compute $E_\infty[h_1^{-1}]$ explicitly through the Milnor-Witt 66-stem.
See \cite{GI} for a chart of this calculation.

The final step is to pass from $E_\infty[h_1^{-1}]$ to $\Ext_A[h_1^{-1}]$
by resolving hidden extensions. 

\begin{prop}
\label{prop:HiddenRels}
Table \ref{ExtRelTable} lists some relations in $\Ext_A$
that are hidden in $E_\infty[h_1^{-1}]$.
Through the Milnor-Witt 66-stem,
all other hidden relations are multiplicative consequences of these relations.
\end{prop}

\begin{pf}
Arguments for the relations involving $g^2B_1$ and $\Delta_1 B_1$ were given already in the proof of Lemma~\ref{lem:d10}. The other relations in Table~\ref{ExtRelTable} are established similarly. 
\end{pf}

Finally, we have calculated $\Ext_A[h_1^{-1}]$ through the
Milnor-Witt 66-stem with the May spectral sequence and obtained the same
answer as in Theorem \ref{ExtAThm}.
Multiplicative generators for $\Ext_A[h_1^{-1}]$ through the
Milnor-Witt 66-stem are given in Table \ref{ExtGenTable}.

\subsection{The $h_1$-local May spectral sequence for $A(2)$}\label{sec:A2MayE1}

We sketch here the calculation of the $h_1$-localized
May spectral sequence over $A(2)$. 
The
$E_1$-term is a polynomial algebra on the generators $h_0$, $h_1$, $h_2$, $h_{20}$, $h_{21}$, and $h_{30}$.

Then $d_1(h_{20}) = h_0 h_1$ and
$d_1(h_{21}) = h_1 h_2$.
As in Section \ref{sec:AMayE1}, we replace $h_{30}$ by $h_{30}' = h_1^{-1} h_0(1) 
= h_{30} + h_1^{-1} h_{20} h_{21}$,
so that $d_1 (h_{30}') = 0$.

The $E_2$-page is then the 
polynomial algebra $\bM_2 [ h_1^{\pm 1}, b_{20}, b_{21}, h_{30}' ]$, and the only differential is
$d_2 ( b_{20} ) = \tau h_1^3$.

It follows that $E_3$ is given by $\F_2 [ h_1^{\pm 1}, b_{20}^2, b_{21}, h_{30}' ]$. No more differentials are possible, and $E_3=E_\infty$.

Note that $b_{20}^2$, $b_{21}$, and $h_{30}'$ correspond
respectively to $v_1^4$, $h_1^{-1} a_1$, and $v_2$ 
in the notation of Proposition \ref{CohomBi(2)}.

\section{The localization map}\label{sec:LocMap}

The calculation of $\Ext_A$ is given in \cite{Istems} up to the $70$-stem. 
In this section, we will use the May spectral sequence analysis of $\Ext_A[h_1^{-1}]$
from Section \ref{sec:MaySS} to determine the localization map
\[ L:\Ext_A \rtarr \Ext_A [h_1^{-1}] \]
in the same range.
A detailed understanding of the localization map is essential for transfer of
information between the localized and non-localized situations.

\begin{prop}\ 
Table~\ref{LocztnTable} lists some values of the localization map
$L: \Ext_A \rightarrow \Ext_A[h_1^{-1}]$ on multiplicative generators of
$\Ext_A$.  Through the 70-stem, 
the localization map is zero on all generators of $\Ext_A$ not listed
in Table~\ref{LocztnTable}.
\end{prop}

\begin{pf}
Note that $\Ext_A[h_1\inv]$ is concentrated in degrees $(t,f,c)$ such that 
$t-c$ is even.
Many of the generators of $\Ext_A$ are in degrees $(t,f,c)$ such that $t-c$ 
is odd.  Therefore, all of these generators must map to $0$ in the localization.

The values of $L$ on $u$, $v$, $u'$, $v'$, and $U$ follow from applying the May $E_4$ relation $h_1^4\Delta = d_0^2 + Pg$ to the May descriptions of these classes. 

The remaining values are again determined by their May descriptions, together with the value of $L(B_1)$, which follows from the relation $h_1^7B_1 = h_1^2e_0^3 + c_0^2 \cdot e_0g$ established in Proposition~\ref{prop:HiddenRels}.
\end{pf}

Table \ref{LocztnTable} gives values for the localization map in two forms.
First, it uses the notation from Theorem \ref{ExtAThm} involving the
elements $v_n$.  Second, it uses a different notation for the generators
of $\Ext_A[h_1^{-1}]$ given in Table \ref{ExtGenTable} that is more compatible
with the standard notation for $\Ext_A$.

With a detailed understanding of the localization map in hand, we can establish
some hidden relations in $\Ext_A$ that are needed in \cite{Istems}.

\begin{cor} The following hidden extensions hold in $\Ext_A$:
\begin{enumerate}
\item $e_0^3 + d_0 \cdot e_0g = h_1^5B_1 $.
\item $d_0 v + e_0 u = h_1^3 x'$.
\item $e_0 u' + d_0 v' = h_1^2 c_0 x'$.
\end{enumerate}
\end{cor}

\begin{pf}
Table \ref{LocztnTable} says that 
$L( e_0^3 + d_0 \cdot e_0g)$ equals $h_1^9 v_3^3 + h_1^9 v_2^2 v_4$, which is non-zero.
It follows that $e_0^3+d_0 \cdot e_0g$ is non-zero in $\Ext_A$. 
From the calculation in \cite{Istems}, 
the only possibility is that it equals $h_1^5 B_1$.
This establishes the first formula.

The argument for the second formula is similar.
Table \ref{LocztnTable} says that 
$L(d_0 v + e_0 u )$ equals
$h_1^6 v_1^4 v_2^2 v_4 + h_1^6 v_1^4 v_3^3$, which is non-zero.
It follows that 
$d_0 v + e_0 u$ is non-zero in $\Ext_A$, and the only possibility is that
it equals $h_1^3 x'$.

For the third formula,
Table \ref{LocztnTable} says that
$L(e_0 u' + d_0 v')$ equals
$h_1^7 v_1^4 v_2 v_3^3 + h_1^7 v_1^4 v_2^3 v_4$, which is non-zero.
It follows that 
$e_0 u' + d_0 v'$ is non-zero in $\Ext_A$.
There are several possible non-zero values for $e_0 u' + d_0 v'$.
However,
$e_0 u' + d_0 v'$ must be annihilated by $\tau$ because both
$u'$ and $v'$ are.
Then  
$h_1^2 c_0 x'$ is the only possible value.
\end{pf}

\subsection{The localization map for $A(2)$}

For completeness, we also describe the localization map 
\[ \Ext_{A(2)}\rtarr \Ext_{A(2)}[h_1^{-1}].
 \]
The calculation of $\Ext_{A(2)}$ is given in \cite{IMotA(2)}.

\begin{prop}
Table \ref{LocztnTable2-A(2)} lists some values of the localization map
$L: \Ext_{A(2)} \rtarr \Ext_{A(2)}[h_1^{-1}]$ on multiplicative generators
of $\Ext_{A(2)}$.  The localization map is zero on all
generators of $\Ext_{A(2)}$ not listed in Table \ref{LocztnTable2-A(2)}.
\end{prop}

\begin{proof}
The generators for $\Ext_{A(2)}$ are given in \cite[Table 7]{IMotA(2)}.
The values of $L$ follow by comparison of the localized and unlocalized
May spectral sequences for $A(2)$.
\end{proof}

Now consider the diagram
\[ \xymatrix{
\Ext_{A} \ar[r] \ar[d] & \Ext_{A}[h_1^{-1}] \ar[d] \\
\Ext_{A(2)} \ar[r] & \Ext_{A(2)}[h_1^{-1}]
} \]
in which the horizontal maps are localizations and the vertical maps are
induced by the inclusion $A(2) \rtarr A$.
Given that $\Ext_A[h_1^{-1}]$ and $\Ext_{A(2)}[h_1^{-1}]$
are computed explicitly in Theorem \ref{ExtAThm} and 
Proposition \ref{CohomBi(2)}, one might expect that the map
$\Ext_A [h_1^{-1}] \rtarr \Ext_{A(2)} [h_1^{-1}]$
would be easy to determine.  
The obvious guess is that this map takes $v_1^4$ to $v_1^4$,
takes $v_2$ to $v_2$, and takes $v_n$ to $0$ for $n \geq 3$.
However, the Cartan-Eilenberg spectral
sequences of Section \ref{sec:CohomA}
hide some of the values of this map.

\begin{lemma}
\label{lem:restriction}
The map $\Ext_A[h_1^{-1}] \rtarr \Ext_{A(2)} [h_1^{-1}]$
takes $v_1^4$, $v_2$, $v_3$, $v_4$, $v_5$, and $v_6$ to
$v_1^4$, $v_2$, $h_1^{-3} a_1 v_2$, 
$h_1^{-9} a_1^3 v_2$, $h_1^{-21} a_1^7 v_2$, and
$h_1^{-45} a_1^{15} v_2$.
\end{lemma}

\begin{proof}
This follows from the May spectral sequence calculations of
Section \ref{sec:MaySS}.  
The given values for 
$\Ext_A[h_1^{-1}] \rtarr \Ext_{A(2)} [h_1^{-1}]$ are
apparent on the May $E_\infty$-pages.
We are using that $a_1$ is represented by $a_1=h_1b_{21}$.
\end{proof}

Lemma \ref{lem:restriction} suggests an obvious conjecture 
on the complete description of the map
$\Ext_A [h_1^{-1}] \rtarr \Ext_{A(2)} [h_1^{-1}]$.

\begin{conj}\label{conj:restriction} The map 
$\Ext_A [h_1^{-1}] \rtarr \Ext_{A(2)} [h_1^{-1}]$
takes $v_1^4$ to $v_1^4$ and takes
$v_n$ to $h_1^{-3(2^{n-2}-1)} a_1^{2^{n-2} - 1} v_2$ for $n \geq 2$.
\end{conj}

\section{The Adams spectral sequence for $S[\eta^{-1}]$}\label{sec:ASS}

Recall that $\Ext_A$ is the $E_2$-page for the motivic Adams spectral sequence that converges to the $2$-complete motivic stable homotopy groups $\pi_{*,*}$ of the motivic sphere $S^{0,0}$. The element $h_1$ in $\Ext_A$ detects the motivic Hopf map $\eta$ in $\pi_{1,1}$. 

\begin{defn}
Let $S^{0,0}[\eta^{-1}]$ to be the homotopy colimit of the sequence
\[ S^{0,0} \xrightarrow{\eta} S^{-1,-1} \xrightarrow{\eta}S^{-2,-2}\xrightarrow{\eta} \dots.\]
\end{defn}

The homotopy groups $\pi_{*,*}(S^{0,0}[\eta^{-1}])$ are then the target of 
an $h_1$-localized Adams spectral sequence whose $E_2$-page is $\Ext_A[h_1^{-1}]$.
However, we must consider convergence.  A priori, there could be an infinite family of homotopy classes linked together by infinitely many hidden $\eta$-multiplications. These classes would not be detected in $\Ext_A[h_1^{-1}]$. But this cannot occur, as the argument of \cite{A} carries over readily to the motivic setting to establish a vanishing line of slope $1$ in $\Ext_A$.

In the $(s,f,w)$-grading, the Adams differentials behave according to
\[ d_r:E_r^{s,f,w} \rtarr E_r^{s-1,f+r,w}.\]
In the $h_1$-invariant grading, this becomes
\[ d_r:E_r^{t,f,c} \rtarr E_r^{t-1,f+r,c+r-1}.\]

\begin{prop}
\label{prop:Adams-d2}
The Adams $d_2$ differential for $S^{0,0}[\eta^{-1}]$ takes the following values.
\begin{enumerate}
\item $d_2(v_1^4) = 0$.
\item $d_2(v_2) = 0$.
\item $d_2(v_3) = h_1v_2^2$.
\item $d_2(v_4) = h_1v_3^2$.
\end{enumerate}
\end{prop}

\begin{proof}
The first two formulas follow immediately because there are no possible
non-zero values.
The third formula follows from the Adams differential $d_2(e_0) = h_1^2 d_0$
 in the unlocalized case \cite{Istems}, together with the fact that
 the localization map takes $c_0$ and $e_0$ to $h_1^2 v_2$ and $h_1^3 v_3$.
The fourth formula follows from the Adams differential $d_2(e_0 g) = h_1^2 e_0^2$
in the unlocalized case \cite{Istems}, together with the fact that
the localization map takes $e_0$ and $e_0 g$ to $h_1^3 v_3$ and $h_1^7 v_4$.
\end{proof}

Proposition \ref{prop:Adams-d2} suggests an obvious conjecture for the
values of the Adams $d_2$ differential on the rest of the generators of
$\Ext_A[h_1^{-1}]$.  See Conjecture \ref{mainConj} for an explicit statement.

\subsection{Motivic modular forms and Adams differentials}
In the classical case, the topological modular forms spectrum $\tmf$
is a spectrum whose $\F_2$-cohomology
is equal to the quotient $A_{\cl}//A(2)_{\cl}$.  This implies that 
$\Ext_{A(2)_{\cl}}$
is the $E_2$-page of the Adams spectral sequence converging to the 
$2$-complete homotopy groups of $\tmf$.

One might speculate that there is a motivic spectrum $\mmf$ (called
``motivic modular forms") whose motivic $\F_2$-cohomology is isomorphic
to $A//A(2)$.  Then $\Ext_{A(2)}$ would be the $E_2$-page of the motivic
Adams spectral sequence converging to the $2$-complete motivic
homotopy groups of $\mmf$.  However, no such motivic spectrum is known 
to exist.  See \cite{NSO} for one piece of the program for constructing $\mmf$.

In any case, we assume for the rest of this section that $\mmf$ does exist,
and we explore some of the computational consequences.

\begin{lemma}[\cite{IMotA(2)},\S4.4]
\label{lem:d2-a1}
Suppose that $\mmf$ exists. Then, in the Adams spectral sequence 
\[
\Ext_{A(2)}\Rightarrow \pi_{*,*}(\mmf),
\]
there is an Adams differential $d_2(a_1)=h_1^2c_0$.
\end{lemma}

\begin{pf}
Since $d_2(e_0) = h_1^2 d_0$ in the Adams spectral sequence for $S^{0,0}$,
it follows that $d_2(e_0) = h_1^2 d_0$ in the Adams spectral sequence
for $\mmf$ as well.

We have the relation $c_0 a_1 = h_1^2 e_0$ in $\Ext_{A(2)}$  \cite{IMotA(2)}.
Therefore, $a_1$ must support a differential, and
$h_1^2 c_0$ is the only possible value.

Note that the element $a_1$ was called $u$ in \cite{IMotA(2)}.
\end{pf}

\begin{prop} Suppose that $\mmf$ exists. Then Conjecture~\ref{mainConj}
is equivalent to Conjecture~\ref{conj:restriction}.
\end{prop}

\begin{pf}
The existence of $\mmf$ ensures that
the Adams $d_2$ differential is compatible with the map
$r: \Ext_A \rtarr \Ext_{A(2)}$, so that $d_2(r(v_n)) = r(d_2(v_n)$. 
For degree reasons, $r(v_n)$ is either equal to 
$h_1^{-3(2^{n-2}-1)} a_1^{2^{n-2}-1} v_2$, or it is zero. 
Also for degree reasons, $d_2(v_n)$ is either equal to
$h_1 v_{n-1}^2$, or it is zero.

Suppose that Conjecture \ref{mainConj} holds.
Then $r(d_2(v_n))$ equals $h_1 r(v_{n-1})^2$.  We may assume by
induction that $r(v_{n-1}^2)$ equals $h_1^{-3(2^{n-2}-2)} a_1^{2^{n-2} -2} v_2^2$.
In particular, this shows that 
$r(d_2(v_n))$ is non-zero.
But $r(d_2(v_n))$ equals $d_2 (r(v_n))$, so $r(v_n)$ must also be non-zero.
This establishes Conjecture \ref{conj:restriction}.

Now suppose that Conjecture \ref{conj:restriction} holds.
Then $d_2(r(v_n))$ is equal to $h_1^{-3(2^{n-2}-2)+1} a_1^{2^{n-2}-2} v_2^2$
because of the differential $d_2(a_1) = h_1^2 c_0$ from Lemma \ref{lem:d2-a1}.
But $d_2(r(v_n))$ equals $r(d_2(v_n))$,
so $d_2(v_n)$ must also be non-zero.
This establishes Conjecture \ref{mainConj}.
\end{pf}


\clearpage
\section{Tables}\label{sec:Tables}


\begin{table}[ht]
\caption{Generators for May $E_2[h_1^{-1}]$ in $t \leq66$}
\label{E2GenTable}
\begin{tabular}{|C|C|C|C|} 
& \text{description} &  (m-f,t,f,c) & d_2  \\ \hline
\altn{h_0} & h_1^{-2}c_0 = h_{30} + h_1^{-1} h_{20}h_{21}  & (2,3,1,1) &  \\
\altn{h_1} & h_3 & (0,3,1,0) &  \\
\altn{h_2} & h_4 & (0,7,1,0) & \\
\altn{b_{20}}  & h_{40}^2+ h_1^{-2} h_{20}^2h_{31}^2 & (6,14,2,2) & \altn{h_0^2}\altn{h_2} \\
\altn{h_3} & h_5 & (0,15,1,0) & \\
\altn{h_0(1)} & h_1^{-1} h_0(1,3) = h_{50} h_3 + h_1^{-1} h_{20}h_{31}h_{23} & (4,18,2,1) & \altn{h_0}\altn{h_2^2} \\
 &  + h_{40} h_{23} +h_1^{-1} h_{20}h_{41}h_3  & & \\ 
\altn{b_{21}} & b_{23} = h_{23}^2 & (2,22,2,0) & \altn{h_1^2h_3 + h_2^3} \\
\altn{b_{30}} & h_{50}^2 + h_1^{-2} h_{20}^2h_{41}^2 & (8,30,2,2) & \altn{h_3b_{20}} \\ 
\altn{h_4} & h_6 & (0,31,1,0) & \\
\altn{h_1(1)} & h_3(1) = h_4h_{33}+h_{23}h_{24} & (2,34,2,0) & \altn{h_1h_3^2} \\
\altn{b_{22}} & b_{24} = h_{24}^2 & (2,46,2,0) & \altn{h_2^2h_4 + h_3^3} \\
\altn{b_{31}} & b_{33} = h_{33}^2 & (4,54,2,0) & \altn{h_4b_{21}+h_2b_{22}} \\
\altn{b_{40}} & h_{60}^2 + h_1^{-2} h_{20}^2h_{51}^2 & (10,62,2,2) & \altn{h_4b_{30}} \\
\altn{h_5} & h_7 & (0,63,1,0) & \\
\hline
b_{20} & b_{20} = h_{20}^2 & (2,2,2,2) & \tau h_1^3 \\
b_{21} & b_{21} = h_{21}^2 & (2,4,2,0) & h_1^2\altn{h_1} \\
b_{31} & b_{31} = h_{31}^2 &  (4,12,2,0) & \altn{h_2}b_{21} \\
b_{41} & b_{41} = h_{41}^2 &  (6,28,2,0) & \altn{h_3}b_{31} \\
b_{51} & b_{51} = h_{51}^2 &  (8,60,2,0) & \altn{h_4}b_{41} \\
\hline
\end{tabular}
\end{table}

\begin{table}[ht]
\caption{Relations for May $E_2[h_1^{-1}]$ in $t \leq 66$}
\label{E2RelTable}
\begin{tabular}{|C|C|C|} 
\text{relation} & (m-f,t,f,c)  \\ \hline
 \altn{h_0h_1} & (2,6,2,1) \\
\altn{h_1h_2} &  (0,10,2,0) \\
\altn{ h_2b_{20}} +  \altn{h_0h_0(1)}  &  (6,21,3,2) \\
\altn{h_2h_3} & (0,22,2,0) \\
\altn{h_2h_0(1)} +  \altn{h_0b_{21}}&  (4,25,3,1) \\
\altn{h_3h_0(1)} & (4,33,3,1) \\
\altn{b_{20}b_{21}} + \altn{h_1^2b_{30}+ h_0(1)^2} &  (8,36,4,2) \\
\altn{h_0h_1(1)} & (4,37,3,1) \\
\altn{ h_3b_{21}} + \altn{h_1 h_1(1)} & (2,37,3,0) \\
\altn{h_3h_4} &  (0,46,2,0) \\
\altn{b_{20}h_1(1) + h_1h_3b_{30}} &  (8,48,4,2) \\
\altn{h_3h_1(1)} + \altn{h_1b_{22}}& (2,49,3,0) \\
\altn{h_0(1)h_1(1)} & (6,52,4,1) \\
\altn{b_{20}b_{22}} + \altn{h_0^2b_{31} + h_3^2b_{30}} & (8,60,4,2) \\
\altn{b_{22}h_0(1)} + \altn{h_0h_2b_{31}} & (6,64,4,1) \\
\altn{h_4h_1(1)} & (2,65,3,0) \\
 \hline
\end{tabular}
\end{table}

\begin{table}[ht]
\caption{Notation for the $h_1$-localized May spectral sequence}
\label{E4NotationTable}
\begin{tabular}{|C|C|C|} 
& \text{description} & (m-f,t,f,c) \\ \hline
P & b_{20}^2 & (4,4,4,4) \\
e_0 & b_{21}h_0(1) & (4,7,4,1) \\
g & b_{21}^2 & (4,8,4,0) \\
B & \altn{b_{20}}b_{21}  + \altn{h_0^2} b_{31} & (8, 18, 4, 2) \\
B_1 & c_0 B & (10, 21, 7, 3) \\
\phi & h_1 b_{21}B 
& (10,22,7,2) \\
\Delta_1 & b_{31}^2 & (8,24,4,0) \\
D_4 & g\altn{h_0(1)}+  \altn{h_0h_2}b_{21}b_{31} & (8,26,6,1)\\
\altn{P} &  \altn{b_{20}^2} & (12,28,4,4)  \\
s_1 & h_1 {h_4^2}b_{21}b_{31}+ h_1 g{b_{23}} + h_1^{-1} {h_3h_5}b_{21}^3 &  (6,30,7,0) \\ 
D_3' & h_1^4 \altn{b_{20}h_0(1)} & (10,32,8,3)  \\
y_{34} & {h_5}^2b_{21} + h_1^2 {h_3(1)} & (2,34,4,0)  \\
y_{35} & {h_4}b_{41} & (6,35,3,0) \\
\altn{d_0} & \altn{h_0(1)^2} & (8,36,4,2) \\
\boldsymbol{\nu} & \altn{h_2b_{30}} & (8,37,3,2)  \\
\altn{e_0} & \altn{b_{20}h_0(1)} & (6,40,4,1) \\
\altn{g} & \altn{b_{21}^2} & (4,44,4,0) \\
y_{45} &  g\boldsymbol{\nu} + \altn{h_0^2h_3}\Delta_1 & (12,45,7,2) \\
y_{60} & h_1^8 \altn{b_{21} b_{30}} + \altn{h_2^2 b_{30}} b_{21} b_{31}  g & (14,60,12,2)  \\
\altn{\Delta} & \altn{b_{30}^2} & (16,60,4,4)  \\
y_{61} & \altn{h_4 h_0(1)} b_{31} +  \altn{h_0b_{22}}  b_{31} & (8,61,5,1)  \\
&  + \altn{h_0h_3^2}  b_{41} +  \altn{h_0b_{31}} b_{21} & \\
y_{64} & \altn{b_{30}h_1(1)}  & (10,64,4,2)  \\
\hline
\end{tabular}
\end{table}

\begin{table}[ht]
\caption{The $h_1$-localized May $d_4$ differential}
\label{d4MayTable}
\begin{tabular}{|C|C|C|} 
 & (m-f,t,f,c) & d_4  \\ \hline
g & (4,8,4,0) & h_1^4 {h_4} \\
\Delta_1 & (8,24,4,0) & g{h_5} \\
\altn{P} &  (12,28,4,4) & \altn{h_0^4h_3} \\
y_{35} & (6,35,3,0) & y_{34} \\
\boldsymbol{\nu} &  (8,37,3,2) & \altn{h_0^2h_3^2} \\
\altn{\Delta} &  (16,60,4,4) & \altn{h_4P} \\
b_{41}\boldsymbol{\nu} & (14,65,5,2) & h_1^2 y_{64} \\
\hline
\end{tabular}
\end{table}

\begin{table}[ht]
\caption{The $h_1$-localized May $d_6$ differential}
\label{d6MayTable}
\begin{tabular}{|C|C|C|} 
 & (m-f,t,f,c) & d_6  \\ \hline
\phi & (10,22,7,2) & h_1 c_0^2h_5 \\
{c_0}g^3 & (14,27,15,1) & h_1^{10} D_4 \\
{h_4}g^3 & (12,31,13,0) &   h_1^7 s_1 \\
{c_0}g\phi & (16,33,14,3) & h_1^7 D_3' \\
{h_4}g\phi & (14,37,12,2) & h_1^9 \altn{d_0} \\
{h_4}gD_4 & (12,41,11,1) & h_1^8 \altn{e_0} \\
{h_4}gs_1 & (10,45,12,0) & h_1^9 \altn{g} \\
{h_4}g^3\Delta_1 & (20,55,17,0) &  h_1^7 s_1\Delta_1 \\
{c_0}g\phi\Delta_1 & (24,57,18,3) & h_1^7 D_3'\Delta_1 \\
g^2y_{45} & (20,61,15,2) & h_1^4y_{60} \\
{h_4}g\phi\Delta_1 & (22,61,16,2) & h_1^9 \Delta_1\altn{d_0} \\
{c_0}g\phi\altn{P}  & (28,61,18,7) & h_1^7 D_3'\altn{P} \\
{h_4}g\phi b_{41} & (20,65,14,2) & h_1^9 b_{41}\altn{d_0} \\
{h_4}gD_4\Delta_1 & (20,65,15,1) &  h_1^8 \Delta_1\altn{e_0} \\
\hline
\end{tabular}
\end{table}

\begin{table}[ht]
\caption{The $h_1$-localized May $d_8$ differential}
\label{d8MayTable}
\begin{tabular}{|C|C|C|} 
 & (m-f,t,f,c) & d_8 \\ \hline
g^2  & (8,16,8,0) & h_1^8 {h_5} \\
\Delta_1^2 & (16,48,8,0) & g^2{h_6} \\
\Delta_1\altn{P} & (20,52,8,4) & h_1^{-6} {c_0}^2e_0^2{h_6} \\
\altn{P}^2 & (24,56,8,8) &  \altn{h_0^8h_4} \\
{c_0}\Delta_1y_{35} & (16,62,10,1) & h_1^6y_{61} \\
\hline
\end{tabular}
\end{table}

\begin{table}[ht]
\caption{The $h_1$-localized May $d_{10}$ differential}
\label{d10MayTable}
\begin{tabular}{|C|C|C|} 
 & (m-f,t,f,c) & d_{10}  \\ \hline
\Delta_1\phi & (18,46,11,2) & h_1^3e_0^2{h_6} \\
\phi\altn{P} & (22,50,11,6) & h_1^{-7}{c_0^6h_6} \\
{c_0}g^6 & (26,51,27,1) & h_1^{18}\Delta_1D_4 \\
g^4 \phi & (26,54,23,2) & h_1^{13}gy_{45} \\
g^2\phi^2 & (28,60,22,4) & h_1^{12}c_0\altn{b_{30}}D_4 \\
g^2\Delta_1\phi & (26,62,19,2) & h_1^{13}\Delta_1\boldsymbol{\nu} \\
\hline
\end{tabular}
\end{table}

\begin{table}[ht]
\caption{The $h_1$-localized May $d_{12}$ differential}
\label{d12MayTable}
\begin{tabular}{|C|C|C|} 
 & (m-f,t,f,c) & d_{12}  \\ \hline
{h_5}g^6 & (24,63,25,0) & h_1^{12}{h_4h_6}g^3 \\
\phi^2 & (20,44,14,4) & h_1^{2}{c_0^4h_6} \\
{c_0}g^4\phi & (28,57,26,3) & h_1^{19}\altn{b_{30}}D_4 \\
e_0g\Delta_1\phi & (26,61,19,3) & h_1^8{h_6}e_0\phi \\
{c_0}g^2\phi^2 & (30,63,25,5) & h_1^{16}\altn{b_{30}}D_3' \\
g^2\Delta_1 D_4 & (24,66,18,1) & h_1^8{h_6}gD_4 \\
\hline
\end{tabular}
\end{table}

\begin{table}[ht]
\caption{The $h_1$-localized May $d_{14}$ differential}
\label{d14MayTable}
\begin{tabular}{|C|C|C|} 
 & (m-f,t,f,c) & d_{14} \\ \hline
g^2\phi & (18,38,15,2) & h_1^{9}{c_0^2h_6} \\
e_0g^6 & (28,55,28,1) & h_1^{21}b_{41}D_4 \\
g^4e_0\phi & (30,61,27,3) & h_1^{18}b_{41}D_3' \\
\hline
\end{tabular}
\end{table}

\begin{table}[ht]
\caption{The $h_1$-localized higher May differentials}
\label{higherMayTable}
\begin{tabular}{|C|C|C|C|} 
 & (m-f,t,f,c) & d_{r} & \text{value}  \\ \hline
g^4 & (16,32,16,0) & d_{16} & h_1^{16}{h_6} \\
e_0^2g^5 & (28,54,28,2) & d_{18} & h_1^{21}{h_6}\phi \\
g^8 & (32,64,32,0) & d_{32} & h_1^{32}{h_7} \\
\hline
\end{tabular}
\end{table}

\begin{table}[ht]
\caption{Hidden Relations}
\label{ExtRelTable}
\begin{tabular}{|C|C|} 
&   (t,f,c) \\ \hline
h_1^2 e_0^3 + c_0^2 \cdot e_0g = h_1^7 B_1 & (21,14,3) \\
h_1^2 e_0(e_0g)^2 + c_0^2\cdot e_0g^3 = h_1^7 g^2B_1 & (37,22,3) \\
h_1^4 e_0^4 \cdot e_0g + c_0^4\cdot e_0g^3 = h_1^{14} B_1\phi & (43,28,5) \\ 
e_0^2 \cdot e_0g^3 + (e_0g)^3 = h_1^{13} \Delta_1 B_1 & (45,24,3) \\
h_1^6 e_0^7 + h_1^4 c_0^2e_0^4 \cdot e_0g + h_1^2 c_0^4e_0(e_0g)^2 + 
c_0^6 \cdot e_0g^3 = h_1^{23} \altn{P}B_1 & (49,34,7) \\
\hline
\end{tabular}
\end{table}

\begin{table}[ht]
\caption{Generators for $\Ext_A$ in $t \leq 66$}
\label{ExtGenTable}
\begin{tabular}{|C|C|C|C|} 
\text{May name} & \text{Theorem \ref{ExtAThm} name} &  (t,f,c) & \text{Adams } d_2\\ \hline
c_0 & h_1^2 v_2 & (3,3,1) &\\
P & v_1^4 & (4,4,4) & \\
e_0 & h_1^3 v_3 & (7,4,1) & c_0^2\\
e_0 g & h_1^7 v_4 & (15,8,1) & h_1^2 e_0^2\\
e_0g^3 & h_1^{15} v_5 & (31,16,1) & h_1^2 (e_0g)^2\ ? \\
e_0g^7 & h_1^{31} v_6 & (63,32,1) & h_1^2(e_0g^3)^2\ ? \\
\hline
\end{tabular}
\end{table}


\begin{landscape}

\begin{table}[ht]
\caption{The localization map $\Ext_A \rtarr \Ext_A[h_1^{-1}]$}
\label{LocztnTable}
\begin{tabular}{CCCHCC}
\text{element} & \text{May description} & (s,f,w) & (s-w,\mathrm{Chow}) & 
\text{Theorem \ref{ExtAThm} value} & \text{Table \ref{ExtGenTable} value} \\
\hline
P^kh_1 & & (1,1,1)+k(8,4,4) & (4k,4k) & h_1v_1^{4k} & h_1P^k \\
P^kc_0 & & (8,3,5)+k(8,4,4) & (3,1) +k(4,4) & h_1^2v_1^{4k}v_2 & P^kc_0 \\ 
P^kd_0 & & (14, 4, 8)+k(8,4,4) & (6,2)+k(4,4) & h_1^2v_1^{4k}v_2^2 & h_1^{-2}P^kc_0^2 \\
P^ke_0 & & (17, 4, 10)+k(8,4,4) & (7,1)+k(4,4) &  h_1^3v_1^{4k}v_3 &  P^ke_0\\
e_0 g & & (37,8,22) & (15,1) & h_1^7v_4 & e_0g\\
u & \Delta h_1 d_0& (39,9,21) & (18,6) & h_1^{3}(v_1^4v_3^2 + v_2^6) & h_1^{-3}(Pe_0^2+c_0^6) \\
P^k u & & (39,9,21)+k(8,4,4) & (18,6)+k(4,4) & h_1^{3}v_1^{4k}(v_1^4v_3^2 + v_2^6) & h_1^{-3}P^k(Pe_0^2+c_0^6) \\ 
v & \Delta h_1 e_0& (42,9,23) & (19,5) & h_1^{4}(v_1^4v_4 + v_2^4v_3) & h_1^{-3}(Pe_0g+h_1^{-4}c_0^4e_0) \\
P^k v & & (42,9,23)+k(8,4,4) & (19,5)+k(4,4) & h_1^{4}v_1^4(v_1^4v_4 + v_2^4v_3) & h_1^{-3}P^k(Pe_0g+h_1^{-4}c_0^4e_0) \\ 
B_1 & c_0B & (46,7,25) & (21,3) & h_1^{4}(v_2^2v_4+v_3^3) & h_1^{-7}c_0^2e_0g + h_1^{-5}e_0^3   \\ 
u' & \Delta c_0 d_0 & (46,11,25) & (21,7) &  h_1^{4}(v_1^4v_2v_3^2 + v_2^7) &  h_1^{-4}(Pc_0e_0^2 + h_1^{-6}c_0^7) \\
P^k u' & & (46,11,25)+k(8,4,4) & (21,7)+k(4,4) & h_1^{4}v_1^{4k}(v_1^4v_2v_3^2 + v_2^7) & h_1^{-4}P^k(Pc_0e_0^2 + h_1^{-6}c_0^7) \\ 
v' & \Delta c_0 e_0& (49,11,27) & (22,6) & h_1^{5}(v_1^4v_2v_4 + v_2^5v_3) & h_1^{-4}(Pc_0e_0g + h_1^{-4}c_0^5e_0)   \\
P^k v' & & (49,11,27)+k(8,4,4) & (22,6)+k(4,4) & h_1^{5}v_1^{4k}(v_1^4v_2v_4 + v_2^5v_3) & h_1^{-4}P^k(Pc_0e_0g + h_1^{-4}c_0^5e_0)\\ 
B_8 & h_0(1)B_1 & (53,9,29) & (24,4) & h_1^{5}(v_2^3v_4+v_2v_3^3) & h_1^{-6}c_0(h_1^{-2}c_0^2e_0g + e_0^3) \\
x' & h_0(1) B P & (53,10,28) & (25,7) & h_1^{3}v_1^4(v_2^2v_4+v_3^3) & h_1^{-6}P(h_1^{-2}c_0^2e_0g+e_0^3) \\
B_{21} & h_0(1)^3 B & (59,10,32) & (27,5) & h_1^{5}(v_2^4v_4+v_2^2v_3^3)  & h_1^{-8}c_0^2(h_1^{-2}c_0^2e_0g + e_0^3)  \\
B_{22} & b_{21}d_0 B & (62,10,34) & (28,4) & h_1^{6}(v_2^2v_3v_4+v_3^4) & h_1^{-6}e_0(h_1^{-2}c_0^2e_0g + e_0^3)   \\
U & \Delta^2 h_1^2 d_0 & (64,14,34) & (30,10) & h_1^{4}(v_1^8v_3v_4+v_2^{10}) & h_1^{-6}(P^2e_0e_0g + h_1^{-10}c_0^{10})   \\
P^2 x' & & (69,18,36) & (33,15) & h_1^{3}v_1^{12}(v_2^2v_4+v_3^3) & h_1^{-6}P^3(h_1^{-2}c_0^2e_0g+e_0^3) \\
\hline
\end{tabular}
\end{table}

\end{landscape}

\begin{table}[ht]
\caption{The localization map for $\Ext_A(2) \rtarr \Ext_{A(2)}[h_1^{-1}]$}
\label{LocztnTable2-A(2)}
\begin{tabular}{CCCHC}
\text{element} & \text{May description} & (s,f,w) & (s-w,\mathrm{Chow}) & \text{value} \\
\hline
P & & (8,4,4) &  & v_1^4 \\
c & & (8,3,5) & & h_1 v_2 \\
u &h_1b_{21} & (11,3,7) & & a_1 \\
d & & (14,4,8) & & h_1^2 v_2^2 \\
e & & (17,4,10) & & a_1 v_2 \\
g & & (20,4,12) & & h_1^{-2} a_1^2 \\
\Delta h_1 & & (25,5,13) & & h_1^{-5} v_1^4 a_1^2 + h_1 v_2^4 \\
\Delta c & & (32,7,17) & & h_1^{-5} v_1^4 v_2 a_1^2 + h_1 v_2^5 \\
\Delta u & & (35,7,19) & & h_1^{-6} v_1^4 a_1^3 + v_2^4 a_1 \\
\Delta^2 & & (48,8,24) & & h_1^{-12} v_1^8 a_1^4 + v_2^8 \\
\hline
\end{tabular}
\end{table}

\end{document}